\documentclass[a4paper,12pt]{article}

\usepackage{cmap}

\usepackage{amsmath,amsthm,amssymb}
\usepackage[T2A]{fontenc}
\usepackage[english]{babel}
\usepackage[cp1251]{inputenc}
\usepackage{amscd}

\usepackage{hyperref}

  \newskip\prethm \prethm3.0pt plus1.3pt minus.4pt
  \newskip\posthm \posthm2.7pt plus1.4pt minus.3pt
  \newtheoremstyle{STATEMENT}%
       {\prethm}{\posthm}{\itshape}{\parindent}{\scshape}
       {.}{.6em plus.2em minus.1em}{}
  \newtheoremstyle{EXPLANATION}%
       {\prethm}{\posthm}{}{\parindent}{\scshape}
       {.}{.6em plus.2em minus.1em}{}
\theoremstyle{STATEMENT}

\newtheorem {theorem}{Theorem}

\newtheorem    {assertion} {Assertion}

\newtheorem    {assumption} {Assumption}

\theoremstyle{EXPLANATION}

\newtheorem   {remark}{Remark}
\newtheorem    {example} {Example}

\begin{document}

\title{Integral affine 3-manifolds.}
\author{Kozlov I.K. \\ MSU, Moscow, Russia. \\ ikozlov90@gmail.com}
\date{}

\maketitle

\begin{abstract}  Affine manifolds are called integral if there is an atlas such that all transition maps are affine transformations with integer matrices of linear parts. In this paper we describe all complete integral affine structures on compact three-dimensional manifolds up to a finite-sheeted covering. Also a complete list of integral affine structures on the three-dimensional torus and compact three-dimensional nilmanifolds was obtained. \end{abstract} 

\begin{center}
\end{center}

\section{Introduction}

An \textit{affine manifold} is a smooth manifold with an atlas  $\left\{ (U_\alpha,\varphi_\alpha) \right\}$ such that all maps $\varphi_\beta^{-1}\varphi_\alpha$ are affine transformations $\vec{x} \mapsto A \vec{x}+\vec{b}$. An affine manifold is called \textit{integral} if the matrices $A$ of all  transition maps are integer, i.e. $A \in \mathrm{GL}_n(\mathbb Z)$.  

Integral affine manifolds naturally arise in symplectic geometry and the theory of integrable Hamiltonian systems. In particular, a base of any Lagrangian bundle with compact connected fibers has a natural integral affine structure (see, for example, \cite{Dui80}, \cite{KozlovIK10}). In action-angle coordinates from the Liouville theorem, this affine structure on the base is defined by the action coordinates.

Earlier, in paper  \cite{KozlovIK10} classifying invariants of Lagrangian bundles were described (this work was based on  \cite{Dui80} and \cite{Mishachev96}). The first invariant is an integral affine structure on the base. The remaining invariants are elements of some cohomology groups, and in \cite{KozlovIK10} it was shown how they can be calculated if one knows the affine structure on the base. Thus, in order to study Lagrangian bundles on (compact) six-dimensional symplectic spaces $\pi: (M^{6}, \omega) \rightarrow B^3$ or integrable systems with $3$ degrees of freedom, a list of possible integral affine structures on (compact) three-dimensional manifolds may be useful. A similar list is described in this paper (see Theorems~\ref{T:IntAff3NilMan} and \ref{T:IntAffSolv3d}).

Previously, two-dimensional Lagrangian bundles over orientable two-dimensional surfaces were classified by K.\,N.~Mishachev in  \cite{Mishachev96}. Moreover, he proved that any two-dimensional integral affine manifold is diffeomorphic to either the torus  $\mathbb{T}^2$ or the Klein bottle $K^2$ and classified all integral affine structures on the two-dimensional torus $\mathbb{T}^2$. Similar results for non-orientable surfaces (i.e., in the case of the Klein bottle $K^2$) were independently (and almost simultaneously) obtained by I.\,K.~Kozlov  \cite{KozlovIK10}  and D.~ Sepe \cite{Sepe}. In this paper, we will continue these studies in the three-dimensional case.

It is well known that compactness of an affine manifold does not imply its (geodesic) completeness. So far, the Markus conjecture, which states that completeness of an affine manifold is equivalent to existence of a parallel volume form (which in turn is equivalent to that for some atlas all the matrices of transition maps $A \in \operatorname{SL}_n(\mathbb{R})$), has not been proven. For simplicity, we will consider only those affine manifolds for which the Markus conjecture holds. We also assume that all manifolds are orientable so that all matrices of transition maps  $A \in \operatorname{SL}_n(\mathbb{Z})$.

\begin{assumption} All integral affine manifolds considered in this paper are assumed to be complete and orientable.\end{assumption}

Any \textit{complete affine manifold} is a quotient $E/\Gamma$, where  $\Gamma \subset \operatorname{Aff}(E)$ is a discrete group of affine transformations that acts freely and properly discontinuously on an affine space $E$. To be short we will often identify a complete affine manifold $M\approx E/ \Gamma$ with the corresponding group  $\Gamma\subset \operatorname{Aff}(E)$.  

All three-dimensional affine crystallographic groups were classified in \cite{Aff3Cryst}. In particular, all compact complete three-dimensional affine manifolds were described there.

\begin{theorem}[\cite{Aff3Cryst}] Let  $M^3$  be a compact three-dimensional manifold. Then the following conditions are equivalent:

\begin{enumerate}

\item $M$ admits a complete affine structure,

\item $M$ is finitely covered by a $\mathbb{T}^2$-bundle over $S^1$;

\item $\pi_1(M)$  is solvable and $M$ is aspherical.

\end{enumerate}

\end{theorem}

A classification of integral affine structures is ``finer'' than the classification of affine crystallographic groups in \cite{Aff3Cryst}. From the algebraic point of view,  \cite{Aff3Cryst} studies conjugacy classes of discrete subgroups $\Gamma \subset \operatorname{Aff}_n(\mathbb{R})$. In this paper we study conjugacy classes of subgroups $\Gamma$ in the smaller group $\operatorname{GL}_n(\mathbb{Z}) \ltimes \mathbb{R}^n$ of affine transformations $\vec{x} \mapsto   A\vec{x} +\vec{b}$ with an integer matrix $A$.

In this paper, we describe all complete integral affine structures on compact three-dimensional manifolds up to a finite-sheeted covering (see Theorem~\ref{T:IntAffSolv3d}). We consider structures up to a covering so that not to get a too big final list. Also in Section  \ref{S:3dNilAffine} we describe all integral affine structures on the three-dimensional torus and compact three-dimensional nilmanifolds (see Theorem  \ref{T:IntAff3NilMan}). Note that Theorem~\ref{T:IntAff3NilMan} is proved by elementary methods without using results from \cite{Aff3Cryst}.

In this paper we use the following notation and definitions. An affine transformation $\vec{x} \mapsto   A\vec{x} +\vec{b}$ is denoted by $(A, \vec{b})$. The linear and vector parts of an affine transformation $g =(A, \vec{b})$ will be denoted by $L(g) =A$ and $v(g) = \vec{b}$ respectively. The group of affine transformations $(A, \vec{b})$ of a real $n$-dimensional space with an integral matrix $A$ is denoted as $\operatorname{AGL}_n( \mathbb Z)$. An affine frame $Oe_1\dots e_n$ of an affine space  $\mathbb{A}^n$ consists of a point $O \in \mathbb{A}^n$ and vectors $e_1 \dots e_n$ that form a basis of $\mathbb{R}^n$. An affine frame will be called \textit{integer} if all vectors $e_i \in \mathbb{Z}^n$. A diffeomorphism of integral affine manifolds  $f: M_1 \to M_2$ that identifies integral affine structures on them will be call an \textit{affine diffeomorphism}.

\paragraph{Acknowledgments.} The author would like to thank the staff of the Department of Differential Geometry and Applications of the MSU Faculty of Mechanics and Mathematics, especially Prof. A.\,A.~Oshemkov for his support in writing this paper. This work was supported by the Russian Science Foundation grant (project № 17-11-01303).

\section{Integral affine 3-nilmanifolds} \label{S:3dNilAffine}

\begin{theorem} \label{T:IntAff3NilMan} Any compact integral affine $3$-dimensional nilmanifold  $M^3$ is complete. Moreover, it is affinely diffeomorphic to a quotient of  $\mathbb{R}^3$ by an action of one of the following groups $\Gamma$ described in Tables  \ref{Tab:IntAff3Torus} and \ref{Tab:IntAff3NilMan}. These tables contain matrices  $A, \vec{b}$ of affine transformations  $x \mapsto  Ax+\vec{b}$ corresponding to generators $g, h, t$ of the group $\Gamma$.

In Series $1$--$4$ ( described in Table \ref{Tab:IntAff3Torus}) the manifold is diffeomorphic to the torus $M^3 \approx \mathbb{T}^3$ and the generators $g, h, t$ commute. In other series (described in Table\ref{Tab:IntAff3NilMan})  they satisfy the relation \begin{equation} \label{Eq:NilManFundGroupRel} [g, t] = [h, t] =e, \qquad [g, h] = t^m, \end{equation}  where $m  \in \mathbb{N}$. In all series $a_1, b_1, b_2,  c_1 \in \mathbb{Z}\backslash\left\{ 0\right\},  b_1^* \in \mathbb{Z}, n \in \mathbb{N}, x_i, y_i, z_i \in \mathbb{R}$. Also, in all series except for $8$ the vector parts $v(g), v(h), v(t)$ must be linearly independent. In Series $8$ the following condition must be satisfied \begin{equation}  \label{Eq:NilManSpecialCond} \left| \begin{matrix} x_3 &  y_1 x_3  - y_3 x_1 + \frac{m}{2n} x_3^2 \\ -n y_3 &  y_2 x_3  - y_3 x_2 + \frac{m n }{2}y_3^2  \end{matrix} \right| \not =0  \end{equation} 

\begin{table}[h!]
\[
\begin{array}{|c|c|c|c|}
\hline
 & g & h & t \\
\hline
1 & E ,\left(\begin{matrix}x_1
\\ y_1 \\ z_1\end{matrix} \right)& E ,\left(\begin{matrix}x_2
\\ y_2 \\ z_2 \end{matrix} \right) & E ,\left(\begin{matrix}x_3
\\ y_3 \\ z_3\end{matrix} \right) \\
\hline
2 & \left(\begin{matrix}1& 0 & b_1
\\ 0&1 & 0 \\ 0 & 0 & 1\end{matrix} \right),\left(\begin{matrix}0
\\ y_1 \\ z_1\end{matrix} \right) & E ,\left(\begin{matrix}x_2
\\ y_2 \\ 0 \end{matrix} \right) & E ,\left(\begin{matrix}x_3
\\ y_3 \\ 0\end{matrix} \right) \\ \hline
3 & \left(\begin{matrix}1& a_1 & 0
\\ 0&1 & 0 \\ 0 & 0 & 1\end{matrix} \right),\left(\begin{matrix}0
\\ y_1 \\ z_1\end{matrix} \right) & \left(\begin{matrix}1& 0 &
n a_1 \\ 0&1 & 0 \\ 0 & 0 & 1\end{matrix} \right),\left(\begin{matrix} x_2
\\ n z_1 \\ z_2 \end{matrix} \right) & E, \left(\begin{matrix}x_3
\\ 0 \\ 0\end{matrix} \right) \\ \hline 
4 & \left(\begin{matrix}1& a_1 & b_1^*
\\ 0&1 & c_1 \\ 0 & 0 & 1\end{matrix} \right),\left(\begin{matrix}0
\\ 0 \\ z_1\end{matrix} \right) & \left(\begin{matrix}1& 0 &
b_2 \\ 0&1 & 0 \\ 0 & 0 & 1\end{matrix} \right),\left(\begin{matrix}x_2
\\ \frac{b_2}{a_1} z_1 \\ 0\end{matrix} \right) & E, \left(\begin{matrix}x_3
\\ 0 \\ 0\end{matrix} \right) \\ \hline 
\end{array} 
\]
\caption{Integral affine 3-tori} 
\label{Tab:IntAff3Torus}
\end{table}

\begin{table}[h!]
\[
\begin{array}{|c|c|c|c|}
\hline
 & g & h & t \\
\hline
5& \left(\begin{matrix}1& a_1 & 0
\\ 0&1 & 0 \\ 0 & 0 & 1\end{matrix} \right),\left(\begin{matrix}0
\\ y_1 \\ z_1\end{matrix} \right) & E ,\left(\begin{matrix}x_2
\\ \frac{m x_3}{a} \\ 0 \end{matrix} \right) & E ,\left(\begin{matrix}x_3
\\ 0 \\ 0\end{matrix} \right) \\ \hline
6& \left(\begin{matrix}1& a_1 & b_1^*
\\ 0&1 & c_1 \\ 0 & 0 & 1\end{matrix} \right),\left(\begin{matrix}0
\\ 0 \\ z_1\end{matrix} \right) & E ,\left(\begin{matrix}x_2
\\ \frac{m x_3}{a} \\ 0 \end{matrix} \right) & E ,\left(\begin{matrix}x_3
\\ 0 \\ 0\end{matrix} \right) \\ \hline
7 & \left(\begin{matrix}1& a_1 & 0
\\ 0&1 & 0 \\ 0 & 0 & 1\end{matrix} \right),\left(\begin{matrix}0
\\ y_1 \\ z_1\end{matrix} \right) & \left(\begin{matrix}1& 0 &
n a_1 \\ 0&1 & 0 \\ 0 & 0 & 1\end{matrix} \right),\left(\begin{matrix} 0
\\ \frac{mx_3}{a_1} + n z_1 \\ z_2 \end{matrix} \right) & E, \left(\begin{matrix}x_3
\\ 0 \\ 0\end{matrix} \right) \\ \hline 
8 & \left(\begin{matrix}1& 0 & 0
\\ 0&1 & c_1 \\ 0 & 0 & 1\end{matrix} \right),\left(\begin{matrix}x_1
\\ y_1 \\ -\frac{mx_3}{nc_1} \end{matrix} \right) & \left(\begin{matrix}1& 0 &
n c_1 \\ 0&1 & 0 \\ 0 & 0 & 1\end{matrix} \right),\left(\begin{matrix} x_2
\\ y_2 \\ \frac{my_3}{c_1} \end{matrix} \right) & E, \left(\begin{matrix}x_3
\\ y_3 \\ 0\end{matrix} \right) \\ \hline 
9 & \left(\begin{matrix}1& a_1 & b_1^*
\\ 0&1 & c_1 \\ 0 & 0 & 1\end{matrix} \right),\left(\begin{matrix}0
\\ 0 \\ z_1\end{matrix} \right) & \left(\begin{matrix}1& 0 &
b_2 \\ 0&1 & 0 \\ 0 & 0 & 1\end{matrix} \right),\left(\begin{matrix}x_2
\\ \frac{mx_3 + b_2 z_1}{a_1} \\ 0\end{matrix} \right) & E, \left(\begin{matrix}x_3
\\ 0 \\ 0\end{matrix} \right) \\ \hline 
\end{array} 
\] \caption{Integral affine 3-nilmanifolds} 
\label{Tab:IntAff3NilMan}
\end{table}

\end{theorem}

\begin{remark} From  Theorem~\ref{T:IntAff3NilMan} one can obtain a classification of  $3$-dimensional integral affine nilmanifolds by a scheme similar to the one described in   \cite{KozlovIK10} and \cite{Mishachev96} in the two-dimensional case. Manifolds from different series are affinely non-diffeomorphic (since they have different linear holonomy).   In one series, one has to to identify the manifolds $M^3$ that are obtained from each other by a change of integer affine frame of $\mathbb R^3$  or by a change of generators of the fundamental group $\pi_1(M^3)$. For example, two  manifolds $M$ and $M'$ from Series $1$ are affinely diffeomorphic if and only if
$C \left(\begin{smallmatrix}x_1 & x_2 & x_3 \\ y_1 & y_2 & y_3 \\ z_1 & z_2 & z_3\end{smallmatrix} \right) D = \left(\begin{smallmatrix}x'_1 & x'_2 & x'_3 \\ y'_1 & y'_2 & y'_3 \\ z'_1 & z'_2 & z'_3\end{smallmatrix} \right)$ for some matrices $C, D \in \operatorname{GL}(3,\mathbb{Z})$. \end{remark}

\begin{proof}[Proof of theorem~\ref{T:IntAff3NilMan}]  Completeness of integral affine nilmanifolds follows from the following previously proved statement.

\begin{theorem}[\cite{AffNilHol}] For a compact affine manifold $ M $ with nilpotent affine holonomy, the following conditions are equivalent: \begin{enumerate} 

\item $M$ is complete;

\item $M$ admits a parallel volume form;

\item the linear holonomy is unipotent;

\item $M$ is a complete affine nilmanifold.

\end{enumerate}  \end{theorem}

Since the linear holonomy is unipotent, there exists a real basis of $\mathbb{R}^3$ in which all matrices  $A$ of linear holonomy are upper unitriangular matrices. The following statement shows that this basis can be taken integer.

\begin{assertion} \label{A:CommIntEigenVect} If all matrices from a certain subset of integer matrices $M \subset \mathrm{GL}_n (\mathbb{Z})$ have a common eigenvector $v$ with eigenvalue $1$, then this vector can be chosen integer. Moreover, this vector $v$ can be taken from some basis $\mathbb{Z}^n$. \end{assertion}

\begin{proof}[Proof of Assertion~\ref{A:CommIntEigenVect}] Indeed, the vector $v$ is given by linear equations  $\left(A-E\right) v =0,$ which must be satisfied for all  $A \in M$. Since all matrices are integer, these linear equations have a common non-zero solution over $\mathbb{R}$ if and only if they have a common non-zero solution over $\mathbb{Q}$. It remains to use the following simple statement.

\begin{assertion}\label{A:IntVectInBasis} An integer vector $v \in \mathbb{Z}^n$  can be included in an integer basis  $\mathbb{Z}^n$ if and only if its coordinates are mutually comprime. \end{assertion}
Thus, any vector from $\mathbb{Q}^n$ is proportional to a vector from some basis $\mathbb{Z}^n$. Assertion~\ref{A:CommIntEigenVect} is proved. \end{proof}

After all the matrices $A$ of the linear holonomy are simultaneously reduced to the upper unitriangular form, the problem is solved by brute force.  Denote by $\operatorname{UT}_{n}(\mathbb{Z})$ the set of upper unitriangular matrices with integer coefficients. We need to find all discrete subgroups  $\Gamma \subset \operatorname{UT}_{3}(\mathbb{Z}) \ltimes \mathbb{R}^n$ that  act on $\mathbb{R}^3$  freely and properly discontinuously.  It is convenient to analyze the cases according to the rank of the subgroup of translations $T$ in the group $\Gamma$.

\begin{enumerate}

\item Let $\operatorname{rk} T =3$. Then the quotient by the action of $T$ is the torus $\mathbb{T}^3$. 
The following simple statement shows that $\Gamma=T$, since there are no nontrivial elements of finite order in $\Gamma/T$.

\begin{assertion} \label{A:QuotCompFiniteGroupQuot} If a group  $G$ acts freely and properly discontinuously on  $\mathbb{R}^n$, then any subgroup $H \subset G$, for which the space  $\mathbb{R}^n/H$ is compact, has finite index: $[G:H]< \infty$. \end{assertion}

Obviously, in this case, we get Series $1$ from Table~\ref{Tab:IntAff3Torus}. 

\item Let $\operatorname{rk} T =2$. In this case, it suffices to add one more element to $T$ so that the quotient by the action of the generated group becomes compact.  First of all, let us specify the form of the added element $g\in \Gamma$. Using Assertion~\ref{A:IntVectInBasis},  it is not hard to prove the following statement.

\begin{assertion} \label{A:UnipotCanon3dForm} Assume that all eigenvalues of an element $g\in \operatorname{AGL}_3( \mathbb{Z})$ are equal to $1$. Then there exists an integer affine frame in which $g$ has one of the following forms: \begin{equation} \label{Eq:AffZUnipoCanonForm} \left( \left( \begin{matrix} 1 & a & b^* \\ 0 & 1 & c \\ 0 & 0 & 1\end{matrix} \right), \left( \begin{matrix} 0 \\ 0 \\ z \end{matrix} \right) \right), \quad \left( \left( \begin{matrix} 1 & 0 & b \\ 0 & 1 & 0 \\ 0 & 0 & 1\end{matrix} \right), \left( \begin{matrix} 0 \\ y \\ z \end{matrix} \right) \right) \quad \text{or} \quad \left( E, \left( \begin{matrix} x \\ y \\ z \end{matrix} \right) \right) \end{equation} where $a, b, b^*, c \in \mathbb{Z}$ and $a, b, c \not =0$. \end{assertion}

Let us note that a change of integer affine frame corresponds to a conjugation by an element of $\operatorname{AGL}_3( \mathbb{Z})$. 

Then, for convenience of verification, let us give a formula for the commutator of two elements $ g_i= \left( \left( \begin{smallmatrix} 1 & a_i & b_i \\ 0 & 1 & c_i \\ 0 & 0 & 1\end{smallmatrix} \right), \left( \begin{smallmatrix} x_i \\ y_i \\ z_i \end{smallmatrix} \right) \right), i=1, 2$: \begin{equation} \label{Eq:UniTrCommutator} \small g_1 g_2 g_1^{-1} g_2^{-1} = \left( \left( \begin{matrix} 1 &  0 &   \left| \begin{matrix} a_1 & c_1 \\ a_2 & c_2 \end{matrix} \right| \\ 0 & 1 &  0\\ 0 & 0 & 1\end{matrix} \right), \left( \begin{matrix}  \left| \begin{matrix} a_1 & y_1 \\ a_2 & y_2 \end{matrix} \right| + \left| \begin{matrix} b_1 & z_1 \\ b_2 & z_2 \end{matrix} \right| - \left| \begin{matrix} a_1 & c_1 \\ a_2 & c_2 \end{matrix} \right| (z_1+z_2) \\  \left| \begin{matrix} c_1 & z_1 \\ c_2 & z_2 \end{matrix} \right| \\ 0 \end{matrix}  \right) \right)  \end{equation} In particular a conjugation of a translation is again a translation: \begin{equation} \label{Eq:TranslAdj} g= \left(A, \vec{b} \right), \quad t = \left(E, \vec{v} \right) \qquad \Rightarrow \qquad g t g^{-1} = \left( E, A \vec{v} \right)\end{equation}

It immediately follows from  \eqref{Eq:TranslAdj} that the group of translations $T \subset \Gamma$ generates a two-dimensional subspace that is invariant with respect to $L(\Gamma)$. Then it is easy to prove the following assertion.

\begin{assertion} \label{A:TwoTransAndUnip}  Consider a group  $\Gamma\subset \operatorname{AGL}_3( \mathbb{Z})$ generated by three affine transformations \[g= \left( \left( \begin{matrix} 1 & a & b \\ 0 & 1 & c \\ 0 & 0 & 1\end{matrix} \right), \left( \begin{matrix} x_1 \\ y_1 \\ z_1 \end{matrix} \right) \right), 
\qquad t_1=\left( E, \left( \begin{matrix} x_2 \\ y_2 \\ z_2 \end{matrix} \right) \right),  
\qquad t_2=\left( E, \left( \begin{matrix} x_3 \\ y_3 \\ z_3 \end{matrix} \right) \right) \]  where $g$ is one of the transformations from \eqref{Eq:AffZUnipoCanonForm} and translations $t_1$ and $t_2$ are linearly independent. The action of $\Gamma$ on $\mathbb{R}^3$ is free and properly discontinuous if and only if the vectors $\left( \begin{smallmatrix} x_2 \\ y_2 \\ z_2 \end{smallmatrix} \right)$ and $\left( \begin{smallmatrix} x_3 \\ y_3 \\ z_3 \end{smallmatrix} \right) $ span an $L(g)$- invariant subspace and  $\det \left( \begin{smallmatrix} x_1 & x_2 & x_3 \\ y_1 & y_2 & y_3  \\ z_1 & z_2 & z_3 \end{smallmatrix}  \right) \not =0$.
\end{assertion}

Further, note that we can change a basis of $T$ so that the relations \eqref{Eq:NilManFundGroupRel} will hold  for $h=t_1, t=t_2$ (there is also a commutative case, when $m=0$). Indeed, according to \eqref{Eq:TranslAdj} 
the subgroup of translations $T$ is $L(g)$-invariant.  An operator $L(g)$ is unipotent, therefore it also acts on $T$ as an unipotent operator. It remains to use the following fact, which follows from Assertion~\ref{A:CommIntEigenVect}.

\begin{assertion}\label{A:SL2Eigen1CanonView} For any unipotent operator $A \in \operatorname{SL}_{2}( \mathbb{Z})$  there exists an integer basis of $\mathbb{Z}^{2}$  in which the matrix  $A$ has the form $\left(\begin{matrix} 1 & m \\ 0 & 1 \end{matrix} \right)$, where $m \in \mathbb Z$, $m \geq 0$. The number $m$ 
is uniquely defined and is equal to the GCD of elements of $A$. \end{assertion}

It is easy to check that after a suitable choice of basis in $T$ for the groups from Assertion~\ref{A:TwoTransAndUnip} we get Series 2 from Table~\ref{Tab:IntAff3Torus} and Series 5 and 6 from Table~\ref{Tab:IntAff3NilMan}. Note that these series are distinguished by the Jordan normal forms of $L(g)$ in  $\mathbb{R}^3$ and when acting on $T$.

Assertion ~\ref{A:TwoTransAndUnip} guarantees that all these series define nilmanifolds. It remains to show that we have described all the series with $\operatorname{rk}T =2$.

\begin{assertion} \label{A:Rk2T3Gen} Let $\operatorname{rk}T =2$. Then there exists $g\in \Gamma$ such that the group $\Gamma$ is generated by $g$ and $T$. \end{assertion} 

\begin{proof}[Proof of Assertion~\ref{A:Rk2T3Gen}]  It suffices to take an element $g$, for which numbers $b_1$ in Series 2 and $a_1$ in Series 5 and 6  are the smallest possible natural numbers. Indeed, if the number 
 $a_1$ (or $b_1$) is the smallest possible, then by multiplying any other element of $g' \in \Gamma$ on $g^k$ we can always make the corresponding coefficient $a'_1$ (or $b'_1$) equal to zero. But according to Assertion~\ref{A:QuotCompFiniteGroupQuot} some degree of the element  $g'g^k \in \Gamma$ must belong to the subgroup generated by $g, t_1$ and $t_2$. This is possible only if  $g' g^k \in T$. Assertion~\ref{A:Rk2T3Gen} is proved. \end{proof}

\item Let $\operatorname{rk} T =1$. In this case, from  \eqref{Eq:TranslAdj} it follows that a vector of translation $\vec{v}$ is a common eigenvector for all operators $L(g)$, where $g\in \Gamma$. This case can be considered analogously to the case $\operatorname{rk} T =2$, but now we have to add two elements $g$ and $h$ to a translation 
$t\in T$. Let us consider several cases.

\begin{enumerate}

\item Suppose that there exists an element  $g \in \Gamma$ such that $L(g)= \left( \begin{smallmatrix} 1 & a_1 & b_1 \\ 0 & 1 & c_1 \\ 0 & 0 & 1\end{smallmatrix} \right)$, where $a_1, c_1\not=0$. By changing coordinates we can assume that $g=\left( \left( \begin{smallmatrix} 1 & a_1 & b_1 \\ 0 & 1 & c_1 \\ 0 & 0 & 1\end{smallmatrix} \right),\left( \begin{smallmatrix} 0 \\ 0 \\ z_1 \end{smallmatrix} \right) \right)$. Since $L(g)$ has only one eigenvector $t= \left(E, \left( \begin{smallmatrix} x_3 \\ 0 \\ 0 \end{smallmatrix} \right) \right)$. Further, without loss of generality, we can assume that $a_1$ is positive and the smallest possible and that the second element has the form  $h = \left( \left( \begin{smallmatrix} 1 & 0 & b_2 \\ 0 & 1 & c_2 \\ 0 & 0 & 1\end{smallmatrix} \right), \left( \begin{smallmatrix} x_2 \\ y_2 \\ z_2 \end{smallmatrix} \right) \right)$. First consider the case $c_2 =0$.

\begin{assertion} \label{A:OneTransAndUnipOneCommEigen}  Consider a group  $\Gamma \subset \operatorname{AGL}_3( \mathbb{Z})$ generated by three affine transformations \[\footnotesize g= \left( \left( \begin{matrix} 1 & a_1 & b_1 \\ 0 & 1 & c_1 \\ 0 & 0 & 1\end{matrix} \right), \left( \begin{matrix} 0 \\ 0 \\ z_1 \end{matrix} \right) \right), 
\quad h = \left( \left( \begin{matrix} 1 & 0 & b_2 \\ 0 & 1 & 0 \\ 0 & 0 & 1\end{matrix} \right), \left( \begin{matrix} x_2 \\ y_2 \\ z_2 \end{matrix} \right) \right),%
\quad t= \left(E, \left( \begin{matrix} x_3 \\ 0 \\ 0 \end{matrix} \right) \right) \]  where $a_1, b_2, c_1, x_3 \not =0$ and $\operatorname{rk}T=1$.  The action of $\Gamma$ on $\mathbb{R}^3$ is free and properly discontinuous if and only if  $z_2 =0, z_1 y_2 x_3 \not =0$ and $\frac{a_1 y_2 - b_2 z_1}{x_3} \in \mathbb{Q}$. The quotient space $\mathbb{R}^3/ \Gamma$ will then be a compact nilmanifold. \end{assertion}

\begin{proof}[Proof of Assertion~\ref{A:OneTransAndUnipOneCommEigen}] Indeed, according to formula  \eqref{Eq:UniTrCommutator} \begin{equation} \label{Eq:CommUniAllUniOneAngle}  [g, h] = \left( E, \left( \begin{matrix} a_1 y_2 + b_1z_2 - b_2 z_1 \\ c_1 z_2  \\ 0 \end{matrix} \right) \right).\end{equation} Since $\operatorname{rk}T=1$,  the vectors $[g, h]$ and $t$ must be commensurable, hence $z_2 =0$ and $\frac{a_1 y_2 - b_2 z_1}{x_3} \in \mathbb{Q}$. Obviously, if $z_1 y_2 x_3 =0$, then the action is not free. It can be explicitly checked that under the specified conditions the action is free and properly discontinuous. It is easy to see that one of the fundamental domains is contained in the  $0 \leq x \leq x_3, 0 \leq y \leq y_2, 0 \leq z \leq z_1$ and therefore $\mathbb{R}^3/ \Gamma$  is compact. Assertion~\ref{A:OneTransAndUnipOneCommEigen} is proved. \end{proof}

Let us now show that the case $c_2 \not =0$ is impossible.

\begin{assertion} \label{A:OneTransAndUnipOneCommEigenBad} A group   $\Gamma \subset \operatorname{AGL}_3( \mathbb{Z})$ generated by three affine transformations  \[\footnotesize g= \left( \left( \begin{matrix} 1 & a_1 & b_1 \\ 0 & 1 & c_1 \\ 0 & 0 & 1\end{matrix} \right), \left( \begin{matrix} 0 \\ 0 \\ z_1 \end{matrix} \right) \right), 
\quad h = \left( \left( \begin{matrix} 1 & 0 & b_2 \\ 0 & 1 & c_2 \\ 0 & 0 & 1\end{matrix} \right), \left( \begin{matrix} x_2 \\ y_2 \\ z_2 \end{matrix} \right) \right),%
\quad t= \left(E, \left( \begin{matrix} x_3 \\ 0 \\ 0 \end{matrix} \right) \right) \]  where  $a_1, c_1, c_2, x_3 \not =0$, and $\operatorname{rk} T =1$, cannot act on $\mathbb{R}^3$ freely and discontinuously. \end{assertion}

\begin{proof}[Proof of Assertion~\ref{A:OneTransAndUnipOneCommEigenBad}] 
According to the formula \eqref{Eq:UniTrCommutator} the linear part of the commutator $L([g, h]) = \left( \begin{smallmatrix} 1 & 0 & a_1 c_2 \\ 0 & 1 & 0 \\ 0 & 0 & 1\end{smallmatrix} \right)$. 
If the group $\Gamma_1$ generated by $g, [g, h]$ and  $t$ acts freely and properly discontinuously, then according to the Assertion~\ref{A:OneTransAndUnipOneCommEigen} the quotient $\mathbb{R}^3/ \Gamma_1$ will be compact. 
But then by Assertion~\ref{A:QuotCompFiniteGroupQuot} some degree of $h$ has to belong to $\Gamma_1$, which is impossible. Assertion~\ref{A:OneTransAndUnipOneCommEigenBad}  is proved. \end{proof}

It is easy to see that if in Assertion~\ref{A:OneTransAndUnipOneCommEigen} we take $t$ as a generator of the  subgroup of translations $T$, then we get Series 4 from Table~\ref{Tab:IntAff3Torus} (if $[g, h]=e$)  or Series 9 from Table~\ref{Tab:IntAff3NilMan}.

\item Assume that linear parts of all elements $g\in \Gamma$ have the form $L(g) =\left(\begin{smallmatrix}1& a_g & b_g \\ 0&1 & 0 \\ 0 & 0 & 1\end{smallmatrix} \right)$.  First of all, let us find out how we can reduce the matrix $X = \left( \begin{smallmatrix} a_g & b_g \\ a_h & b_h \end{smallmatrix} \right) $ for a given pair of elements $g,  h \in \Gamma$. Since we can change the elements  $g, h$ and take other elements of the basis $e_2, e_3$, the matrix $X$ is defined up to the left-right $\operatorname{GL}(2, \mathbb{Z}) \times \operatorname{GL}(2, \mathbb{Z}) $ action.  The following statement describes the orbits of this action on $\operatorname{Mat}_{2 \times 2}\left( \mathbb{Z}\right)$.

\begin{assertion} \label{A:LeftRigthGL2ZOrb} For any integer matrix $\left( \begin{smallmatrix} a & b \\ c & d \end{smallmatrix} \right) \in \operatorname{Mat}_{2\times 2}(\mathbb{Z})$ there exist $C, D \in \operatorname{GL}(2, \mathbb{Z})$ such that \[ C\left( \begin{matrix} a & b \\ c & d \end{matrix} \right) D = \left( \begin{matrix} n & 0 \\ 0 & kn \end{matrix} \right),\] where  $n = \text{GCD} (a,b, c,d)$ and $k = \frac{\left| ad - bc\right|}n^2$. Here we formally assume that $k=n=0$ if $a=b=c=d=0$. \end{assertion}

\begin{proof}[Proof of Assertion~\ref{A:LeftRigthGL2ZOrb}] Among points  $\left( \begin{smallmatrix} a' & b' \\ c' & d' \end{smallmatrix} \right)$  of the orbits of the original matrix let us choose the one such that $a'$ is the smallest possible natural number. We claim that  $a' = n$. Indeed, using the Euclidean algorithm, one can always find  matrices $C$ and $D$ such that  $C \left( \begin{smallmatrix} a' \\ c' \end{smallmatrix} \right) = \left( \begin{smallmatrix} \text{GCD} (a',c') \\ 0 \end{smallmatrix} \right)$ and $\left( \begin{smallmatrix} a' & b' \end{smallmatrix} \right) D = \left( \begin{smallmatrix} \text{GCD} (a',b') & 0 \end{smallmatrix} \right) $.  Therefore, since $a'$  is the smallest possible, we can assume that $b'=c'=0$. Then, for the same reasoning $a' \bigr| d'$ in the matrix $\left( \begin{smallmatrix} a' & 0 \\ 0 & d' \end{smallmatrix} \right)$.

Then $a'$ is the $\text{GCD}$ of the elements of the final matrix.  But the left-right action does not change the $\text{GCD}$ (since it obviously does not decrease under multiplication on $C$ and $D$, and since these  matrices are invertible). Therefore, in the final matrix $a'=n$.  It remains to note that we can assume $d' = kn$, since under the left-right action, the determinant is preserved up to sign. Assertion~\ref{A:LeftRigthGL2ZOrb} is proved. \end{proof}

Thus, we can assume that the matrix  $\left( \begin{smallmatrix} a_g & b_g \\ a_h & b_h \end{smallmatrix} \right) $ has the form $\left( \begin{smallmatrix} a_1 & 0 \\ 0 & n a_1 \end{smallmatrix} \right) $. By changing the origin, we can always make the corresponding components of the translation equal to zero $x_1= y_1 =0$. The following statement is proved analogously to Assertion~\ref{A:OneTransAndUnipOneCommEigen}.

\begin{assertion} \label{A:OneTransAndUnipOneCommEigenCZero}  Consider a group  $\Gamma \subset \operatorname{AGL}_3( \mathbb{Z})$ generated by three affine transformations \[\footnotesize g= \left( \left( \begin{matrix} 1 & a_1 & 0 \\ 0 & 1 & 0 \\ 0 & 0 & 1\end{matrix} \right), \left( \begin{matrix} 0 \\ y_1 \\ z_1 \end{matrix} \right) \right), 
\quad h = \left( \left( \begin{matrix} 1 & 0 & n a_1 \\ 0 & 1 & 0 \\ 0 & 0 & 1\end{matrix} \right), \left( \begin{matrix} 0 \\ y_2 \\ z_2 \end{matrix} \right) \right),%
\quad t= \left(E, \left( \begin{matrix} x_3 \\ 0 \\ 0 \end{matrix} \right) \right) \]  where $a_1, x_3 \not =0$ and $\operatorname{rk} T=1$. The action of  $\Gamma$ on $\mathbb{R}^3$ is free and properly discontinuous if and only if the vector parts of the elements $v(g), v(h)$ and $v(t)$  are linearly independent and $\frac{a_1 y_2 - n a_1 z_1}{x_3} \in \mathbb{Q}$.  Then, the quotient space $\mathbb{R}^3/ \Gamma$ is a compact nilmanifold. \end{assertion}

Thus, we obtain Series  3 from Table~\ref{Tab:IntAff3Torus} (if $[g, h]=e$) or Series 7 from Table~\ref{Tab:IntAff3NilMan}.

\item It remains to consider the case when the linear parts of all elements have the form $L(g) =\left(\begin{smallmatrix}1& 0 & b_g \\ 0&1 & c_g \\ 0 & 0 & 1\end{smallmatrix} \right)$.  
 Similarly to the previous cases, we can reduce the matrix $\left( \begin{smallmatrix} b_g & b_h \\ c_g & c_h \end{smallmatrix} \right) $ to $\left( \begin{smallmatrix} 0 & n c_1 \\ c_1 & 0\end{smallmatrix} \right) $  and after that bring the elements to the form from Series $8$ from Table~\ref{Tab:IntAff3NilMan}.  It remains to show that the action is free and properly discontinuous, and the quotient is compact exactly when \eqref{Eq:NilManSpecialCond} holds.  For that it is convenient to switch from $\operatorname{AGL}_3( \mathbb{Z})$  to the group of all affine transfromations  $\operatorname{Aff}_3(\mathbb{R})$  and take a vector of translation as the first basis vector.

\begin{assertion} \label{A:OneTransRealCase}  Consider a group  $\Gamma \subset \operatorname{Aff}_3( \mathbb{R})$ generated by three affine transformations  \[ \footnotesize g= \left( \left( \begin{matrix} 1 & 0 & b_1 \\ 0 & 1 & c_1 \\ 0 & 0 & 1\end{matrix} \right), \left( \begin{matrix} x_1 \\ y_1 \\ z_1 \end{matrix} \right) \right), 
\quad h = \left( \left( \begin{matrix} 1 & 0 & b_2 \\ 0 & 1 & c_2 \\ 0 & 0 & 1\end{matrix} \right), \left( \begin{matrix} x_2 \\ y_2 \\ z_2 \end{matrix} \right) \right),%
\quad t= \left(E, \left( \begin{matrix} x_3 \\ 0 \\ 0 \end{matrix} \right) \right) \]  where $b_i, c_i \in \mathbb{R}$, $(c_1, c_2) \not =(0, 0), (z_1, z_2) \not =(0, 0)$, $x_3 \not =0$ and $\operatorname{rk} T=1$.  The action of $\Gamma$ on $\mathbb{R}^3$ is free and properly discontinuous if and only if  $c_2 z_1 = c_1 z_2$, $\frac{b_1 z_2 -b_2 z_1}{x_3}  \in \mathbb{Q}$ and \begin{equation} \label{Eq:OneTransRealCaseCond}  \left| \begin{matrix} c_1 & y_1 - \frac{c_1 z_1}{2} \\ c_2 & y_2 - \frac{ c_2 z_2}{2} \end{matrix} \right| \not =0. \end{equation} The quotient space $\mathbb{R}^3/ \Gamma$ is then a compact nilmanifold. \end{assertion}

\begin{proof}[Proof of Assertion~\ref{A:OneTransRealCase}  ] Indeed, according to \eqref{Eq:UniTrCommutator} \begin{equation} \label{Eq:CommUniAllUniOneAngle}  [g, h] = \left( E, \left( \begin{matrix} b_1 z_2  - b_2 z_1 \\ c_1 z_2 - c_2 z_1 \\ 0 \end{matrix} \right) \right).\end{equation}  Since $\operatorname{rk}T=1$ the vectors $[g, h]$ and $t$ must be commensurable and we get  $c_2 z_1 = c_1 z_2$ and $\frac{b_1 z_2 -b_2 z_1}{x_3} \in \mathbb{Q}$. Since $(c_1, c_2)\not =(0, 0)$, there exists $u \in \mathbb{R}$ such that  $z_1 =c_1 u, z_2 = c_2 u$.  Moreover,
$u \not =0$, since $(z_1, z_2) \not =(0, 0)$. The remaining condition \eqref{Eq:OneTransRealCaseCond} 
 is equivalent to the fact that the matrices  \[ \footnotesize g= \left( \left( \begin{matrix}  1 & c_1 \\  0 & 1\end{matrix} \right), \left( \begin{matrix}  y_1 \\c_1u \end{matrix} \right) \right), 
\quad h = \left( \left( \begin{matrix}  1 & c_2 \\  0 & 1\end{matrix} \right), \left( \begin{matrix}  y_2 \\ c_2 u \end{matrix} \right) \right),\] generates a discrete subgroup of rank $ 2 $ in the commutative two-dimensional group  \[\left( \left( \begin{matrix}  1 & c \\  0 & 1\end{matrix} \right), \left( \begin{matrix}  \frac{c^2 u}{2} +t \\ c u \end{matrix} \right) \right), \qquad c, t \in \mathbb{R}\] which acts freely and transitively on the plane.  It is precisely the case when the action is free and properly discontinuous  and the quotient space $\mathbb{R}^3/\Gamma$ is compact. Assertion~\ref{A:OneTransRealCase} is proved. \end{proof}

It can be easily seen that condition \eqref{Eq:OneTransRealCaseCond} for Series 8 from Table~\ref{Tab:IntAff3NilMan}  can be rewritten as \eqref{Eq:NilManSpecialCond}, which had to be shown.

\end{enumerate}

\item Let $\operatorname{rk} T =0$. It remains to prove the following assertion.

\begin{assertion} \label{A:NoUniRkTZero} Affine holonomy of a compact integral $3$-dimensional affine nilmanifold contains at least one nontrivial translation.  In other words, for the corresponding subgroup $\Gamma \subset \operatorname{AGL}_3( \mathbb{Z})$ the translation subgroup is nontrivial: $\operatorname{rk} T \not =0$.\end{assertion}

\begin{proof}[Proof of Assertion~\ref{A:NoUniRkTZero} ]  On the contrary, assume that there is a group $\Gamma$, 
that acts freely and properly discontinuously on $\mathbb{R}^3$ so that $\mathbb{R}^3/\Gamma$  is compact and  $\operatorname{rk} T  =0$. Without loss of generality, we can assume that the generators of $\Gamma$ have the form \[\footnotesize g= \left(\left(\begin{matrix}1& a_1 & b_1
\\ 0&1 & c_1 \\ 0 & 0 & 1\end{matrix} \right),\left(\begin{matrix}0
\\ 0 \\ z_1\end{matrix} \right) \right), \quad %
h=\left(\left(\begin{matrix}1& 0 &
b_2 \\ 0&1 & c_2 \\ 0 & 0 & 1\end{matrix} \right),\left(\begin{matrix}x_2
\\ y_2 \\ z_2 \end{matrix} \right)\right), \quad %
t=\left(\left(\begin{matrix}1& 0 &
b_3 \\ 0&1 & 0 \\ 0 & 0 & 1\end{matrix} \right),\left(\begin{matrix}x_3
\\ y_3 \\ z_3 \end{matrix} \right)\right) \] where $a_1, c_1, c_2, b_3, z_1 \not =0$ and components $a_1, c_2, b_3$  are the smallest possible natural numbers. It is easy to check that if we take out any of the elements $g, h, t$, then the quotient $\mathbb{R}^3/\Gamma$ cannot be compact. Since $\operatorname{rk} T  =0$ we have $[h, t]=e$. Therefore, using formula \eqref{Eq:UniTrCommutator} (or \eqref{Eq:CommUniAllUniOneAngle}) we get \[ -c_2 z_3 =0, \qquad b_2 z_3 - b_3 z_2 =0 \qquad \Rightarrow \qquad z_2 =z_3 =0\] Analogously $[g, t]=e$ and thus $a_1 y_3 = b_3 z_1.$ Again, using the formula \eqref{Eq:UniTrCommutator} we get \[ [g, h] = \left( \begin{matrix} 1 & 0 & a_1 c_2 \\ 0 & 1 & 0 \\ 0 & 0 & 1\end{matrix} \right) \left( \begin{matrix} a_1 y_2 + b_1 z_2 - b_2 z_1 - a_1 c_2 (z_1+z_2) \\ c_1 z_2 - c_2 z_1 \\ 0 \end{matrix}  \right) \] Since $b_3$ is the smallest possible $[g, h]=t^m$ holds. But then \[ a_1 c_2 = m b_3, \qquad  -c_2 z_1 = my_3 =  m \frac{b_3 z_1}{a_1}.\] Thus $a_1c_2 z_1= -a_1c_2 z_1$, but $a_1 c_2 z_1 \not =0$, and we get a contradiction. Assertion~\ref{A:NoUniRkTZero} is proved.  \end{proof}

\end{enumerate}

Theorem \ref{T:IntAff3NilMan} is proved. \end{proof}

\section{Integral affine 3-solvmanifolds} \label{S:Solv3dAff}

Let us describe the remaining complete integral affine three-dimensional manifolds that are not covered by nilmanifolds. All of them contain hyperbolic elements in the group $\Gamma$ and generalize the following simple example.

\begin{example} \label{Ex:IntAffTorusMap} For any matrix $A = \left(\begin{smallmatrix} a & b \\ c & d \end{smallmatrix} \right) \in \operatorname{SL}_{2}( \mathbb{Z})$ the quotient of $\mathbb{R}^3$ by the group $\Gamma$ generated by elements \[g= \left( \left( \begin{matrix} 1 & 0 & 0 \\ 0 & a & b \\
0 & c & d \end{matrix} \right), \left(
\begin{matrix} 1 \\ 0 \\ 0 \end{matrix} \right) \right) , \quad t_1 = \left( \left( \begin{matrix} 1 & 0 & 0 \\ 0 & 1 & 0 \\
0 & 0 & 1 \end{matrix} \right), \left(
\begin{matrix} 0 \\ 1 \\ 0 \end{matrix} \right) \right) \quad  t_2 = \left( \left( \begin{matrix} 1 & 0 & 0 \\ 0 & 1 & 0 \\
0 & 0 & 1 \end{matrix} \right), \left(
\begin{matrix} 0 \\ 0 \\ 1 \end{matrix} \right) \right)\]  
is an integral affine $3$-dimensional manifold $M^3$. Moreover,  $M^3$  is diffeomorphic to a  $\mathbb{T}^2$-bundle over $S^1$ with monodromy matrix $A$. \end{example}

\begin{remark} Note that there are naturally defined subsets in affine manifolds with transformation maps from $\operatorname{GL}_n(\mathbb{Z}) \ltimes \mathbb{Z}^n$ (that is, all their elements are integers).  For example, these are points that have integer coordinates in any chart (or rational coordinates from $\frac{1}{k} \mathbb{Z}^n$). \end{remark}

\begin{remark} It is easy to see that two manifolds from Example ~\ref{Ex:IntAffTorusMap} are affinely diffeomorphic if and only if the corresponding matrices  $A_1$ and $A_2$ lie in the same conjugacy class in the group 
$\operatorname{SL}_{2}( \mathbb{Z})$.  Description of such conjugacy classes is a well-studied question (see, for example, \cite{Karpenkov13}).  \end{remark}

If a matrix $A\in \operatorname{SL}_{2}( \mathbb{Z})$ is not unipotent, then in some basis it has the form  $ \left( \begin{smallmatrix} e^s & 0 \\ 0 & e^{-s} \end{smallmatrix} \right)$. The groups from Example~\ref{Ex:IntAffTorusMap} with non-unipotent  matrix $A$  correspond to the next series of affine crystallographic groups from \cite{Aff3Cryst}.

\begin{example} \label{Ex:SolvAffGeneral} For any  $\lambda \in \mathbb{R}$  the group \begin{equation} \label{Eq:HyperGroup3dAffCryst}  I_{\lambda} = \left\{  \left( \left( \begin{matrix}  1& \lambda e^{s}u  & \lambda e^{-s}t \\ 0 & e^{s} & 0 \\ 0 & 0 & e^{-s}  \end{matrix} \right)   \left( \begin{matrix} s+ \lambda ut \\ t \\ u  \end{matrix} \right) \right), \qquad s, t, u \in \mathbb{R} \right\} \end{equation} 
acts freely and transitively on  $\mathbb{R}^3$. As a consequence, for any discrete subgroup $\Gamma \subset I_{\lambda}$ the quotient $\mathbb{R}^3/\Gamma$ is an affine manifold. It is easy to check that if the quotient  $\mathbb{R}^3/\Gamma$  is compact, then the group $\Gamma$ possesses generators of the form \begin{equation} \label{Eq:SolvAffGenDisGr}   \begin{gathered} g= \left( \left( \begin{matrix} 1 & \lambda e^s z_0  & \lambda e^{-s} y_0 \\ 0 & e^s & 0 \\
0 & 0 & e^{-s} \end{matrix} \right), \left(
\begin{matrix} s+ \lambda z_0 y_0 \\ y_0 \\ z_0  \end{matrix} \right) \right) , \\
h_1 = \left( \left( \begin{matrix} 1 & \lambda z_1 & \lambda y_1 \\ 0 & 1 & 0 \\
0 & 0 & 1 \end{matrix} \right), \left(
\begin{matrix} \lambda y_1 z_1 \\ y_1 \\ z_1 \end{matrix} \right) \right) \quad  h_2 = \left( \left( \begin{matrix} 1 & \lambda z_2 & \lambda y_2 \\ 0 & 1 & 0 \\ 0 & 0 & 1 \end{matrix} \right), \left( \begin{matrix} \lambda y_2 z_2 \\ y_2 \\ z_2 \end{matrix} \right) \right),  \end{gathered} \end{equation} where $\operatorname{det} \left( \begin{smallmatrix} y_1 & z_1 \\ y_2 & z_2 \end{smallmatrix} \right) \not =0 $ and $\left( \begin{smallmatrix} e^s & 0 \\ 0 & e^{-s} \end{smallmatrix} \right) \left( \begin{smallmatrix} y_1 & z_1 \\ y_2 & z_2 \end{smallmatrix} \right)= \left( \begin{smallmatrix} y_1 & z_1 \\ y_2 & z_2 \end{smallmatrix} \right) A$ for some matrix $A \in \operatorname{SL}_{2}( \mathbb{Z})$. \end{example}

If $\lambda =0$,  then the group  $\Gamma$ from Example~\ref{Ex:SolvAffGeneral}  is generated by three elements, two of which are translastions. Using Assertion~\ref{A:CommIntEigenVect}, it is easy to check that any such subgroup  $\Gamma \subset  \operatorname{AGL}_3( \mathbb Z)$ is conjugate in $\operatorname{AGL}_3( \mathbb Z)$ to one of the following subgroups.

\begin{assertion} \label{A:AffHyper1}  Consider a group  $\Gamma\subset \operatorname{AGL}_3( \mathbb{Z})$, generated by three affine transformations \[g= \left( \left( \begin{matrix} 1 & p & q \\ 0 & a & b \\ 0 & c & d\end{matrix} \right), \left( \begin{matrix} x_1 \\ 0 \\ 0 \end{matrix} \right) \right), 
\qquad t_1=\left( E, \left( \begin{matrix} x_2 \\ y_2 \\ z_2 \end{matrix} \right) \right),  
\qquad t_2=\left( E, \left( \begin{matrix} x_3 \\ y_3 \\ z_3 \end{matrix} \right) \right), \] where $a +d>2$, the vector parts $v(g), v(t_1)$ and $v(t_2)$ are linearly independent and $ \left( \begin{smallmatrix} 1 & p & q \\ 0 & a & b \\ 0 & c & d\end{smallmatrix} \right) \left( \begin{smallmatrix} x_2 & x_3 \\ y_2 & y_3 \\ z_2 & z_3\end{smallmatrix} \right) =\left( \begin{smallmatrix} x_2 & x_3 \\ y_2 & y_3 \\ z_2 & z_3\end{smallmatrix} \right)  A  $ for some matrix $A \in \operatorname{SL}_{2}( \mathbb{Z})$.Then the quotient $\mathbb{R}^3/\Gamma$ is an integral affine $3$-dimensional manifold $M^3$ that is diffeomorphic to a $\mathbb{T}^2$-bundle over $S^1$ with monodromy matrix  $A$.
\end{assertion}

It remains to analyze the case when $\Gamma \subset \operatorname{AGL}_3( \mathbb{Z})$  is conjugate in $\operatorname{Aff}_3(\mathbb{R})$ to a discrete subgroup in  $I_{\lambda}$ and $\lambda \not =0$. Using Assertions~\ref{A:CommIntEigenVect} and \ref{A:LeftRigthGL2ZOrb} we can simplify the linear parts of generators of $\Gamma$. Moreover, by choosing a suitable origin, we can replace a conjugation by an element from $\operatorname{Aff}_3(\mathbb{R})$ with conjugation by an element from $\operatorname{GL}_3(\mathbb{R})$. Thus, we obtain the following statement.

\begin{assertion} \label{A:AffHyper2}  Assume that a subgroup $\Gamma \subset \operatorname{AGL}_3( \mathbb{Z})$ is conjugate in  $\operatorname{Aff}_3(\mathbb{R})$ to some discrete subgroup in the group $I_{\lambda}$ (given by \eqref{Eq:HyperGroup3dAffCryst}), where $\lambda \not =0$. Then there is an integer affine frame in which  $\Gamma$ is generated by elements of the form
\[g= \left( \left( \begin{smallmatrix} 1 & p & q \\ 0 & a & b \\ 0 & c & d\end{smallmatrix} \right), \left( \begin{smallmatrix} x_1 \\ y_1 \\ z_1 \end{smallmatrix} \right) \right), 
\quad h_1=\left(  \left( \begin{smallmatrix} 1 & a_1 & 0 \\ 0 & 1 & 0 \\
0 & 0 & 1 \end{smallmatrix} \right), \left( \begin{smallmatrix} x_2 \\ y_2 \\ z_2 \end{smallmatrix} \right) \right),  
\quad h_2=\left(  \left( \begin{smallmatrix} 1 & 0 & na_1 \\ 0 & 1 & 0 \\
0 & 0 & 1 \end{smallmatrix} \right), \left( \begin{smallmatrix} x_3 \\ y_3 \\ z_3 \end{smallmatrix} \right) \right), \]
where $a_1, n \in \mathbb{N}$ and $\left( \begin{smallmatrix} a_1 & 0\\ 0 & na_1\end{smallmatrix} \right) \left( \begin{smallmatrix} a & b \\ c & d \end{smallmatrix} \right) = A \left( \begin{smallmatrix} a_1 & 0\\ 0 & na_1\end{smallmatrix} \right)$ for some matrix $A \in \operatorname{SL}_{2}( \mathbb{Z})$. In this case, generators $g, h_1, h_2$ have the form \eqref{Eq:SolvAffGenDisGr} when conjugated by some matrix $\left( \begin{smallmatrix} \alpha & u \\ 0 & P \end{smallmatrix} \right)$, where $\alpha \in \mathbb{R}^*, u \in \operatorname{Mat}_{1\times 2}(\mathbb{R})$ and $P \in \operatorname{GL}_{2}(\mathbb{R})$. \end{assertion}

\begin{remark}The condition $P A P^{-1} =\left( \begin{smallmatrix} e^s & 0 \\ 0 & e^{-s} \end{smallmatrix} \right)$ defines the matrix $P$  up to multiplication by a constant. If we fix the matrix $P$, then for any  $\alpha \in \mathbb{R}^*, u \in \operatorname{Mat}_{1\times 2}(\mathbb{R})$ the condition of conjugacy of $g, h_1, h_2$ to the elements  \eqref{Eq:SolvAffGenDisGr} provides restrictions on the coefficients of the vector parts $x_i, y_i, z_i$. \end{remark}

Now let us formulate the main result of this paper. It is convenient to divide integral affine $3$-dimensional manifolds into three classes according to their type of geometry (for more information about geometries on $3$-dimensional manifolds, see \cite{Thur01} or \cite{Skott86}).

\begin{theorem}\label{T:IntAffSolv3d} Any compact integral affine $3$-dimensional manifold $M^3$ has a geometric structure modeled on $E^3, \operatorname{Nil}$ or $\operatorname{Sol}$. The geometric structure on $M$ is completely determined by $\pi_1(M)$.

\begin{enumerate}

\item If  $\pi_1(M)$ is virtually Abelian, then  $M^3$ is a Euclidean manifold. Moreover,  $M^3$ is finitely covered by  one of the manifolds from Theorem~\ref{T:IntAff3NilMan} described in Table~\ref{Tab:IntAff3Torus}. 

\item If $\pi_1(M)$ is virtually nilpotent, but not virtually Abelian, then  $M^3$ is a $\operatorname{Nil}$-manifold. Moreover, $M^3$ is finitely covered by  one of the manifolds from Theorem~\ref{T:IntAff3NilMan}, described in Table ~\ref{Tab:IntAff3NilMan}. 

\item If $\pi_1(M)$ is solvable, but not virtually nilpotent, then $M^3$ is a $\operatorname{Sol}$-manifold. Moreover,  $M^3$  is finitely covered by an integral affine manifold $\mathbb{R}^3/\Gamma$  for one of the groups $\Gamma$,  described in the Assertions~\ref{A:AffHyper1}  or \ref{A:AffHyper2}.

\end{enumerate}

\end{theorem}

\begin{proof}[Proof of Theorem~\ref{T:IntAffSolv3d}] According to \cite{Aff3Cryst} any compact $3$-dimensional affine manifold is finitely covered by a solvmanifold. By taking a finite cover we can assume that the algebraic closure of $\bar{\Gamma}$ is connected, and therefore, by the Lie--Kolchin theorem, there exists a basis (over $\mathbb{C}$), in which all matrices from $L(\Gamma)$ are upper triangular. Moreover, up to a finite cover, we can assume that all eigenvalues are real and positive.

Thus, we can assume that $M^3 \approx \mathbb{R}^3/\Gamma$ are solvmanifolds, and that the linear parts of all elements  $L(g)$  are either unipotent or have Jordan normal form  $\left(\begin{smallmatrix} 1 & 0 & 0 \\ 0 & e^s & 0 \\ 0 & 0 & e^{-s} \end{smallmatrix} \right)$. If all $L(g)$ are unipotent, then $M^3$  is a nilmanifold, and this case is analyzed in Theorem~\ref{T:IntAff3NilMan}. Otherwise, according to \cite{Aff3Cryst}, the group $\Gamma$ is conjugate (in $\operatorname{Aff}_3(\mathbb{R})$) to a discrete subgroup of the group $I_{\lambda}$ (see Example~\ref{Ex:SolvAffGeneral}).  And, therefore, $\Gamma$  is conjugate in  $\operatorname{AGL}_3(\mathbb{Z})$ to one of the groups from Assertions~\ref{A:AffHyper1}  or \ref{A:AffHyper2}. Theorem~\ref{T:IntAffSolv3d} is proved.\end{proof}

\end{document}